\documentclass[11pt]{amsart}

\usepackage{amsmath,amssymb,amscd,amsfonts,verbatim}
\usepackage{paralist}
\usepackage[mathscr]{eucal}
\usepackage{color}

\usepackage[pdftex]{hyperref}
\hypersetup{citecolor=blue,linktocpage}

\newtheorem{thm}{Theorem}[section]

\newtheorem{lem}[thm]{Lemma}

\newtheorem{prop}[thm]{Proposition}
\newtheorem{cor}[thm]{Corollary}

\newtheorem{assu-nota}[thm]{Assumption--Notation}

\theoremstyle{remark}
\newtheorem{ex}[thm]{Example}
\newtheorem{remark}{Remark}

\newcommand{\C}{\mathbb C}
\newcommand{\Z}{\mathbb Z}
\newcommand{\Q}{\mathbb Q}

\newcommand{\pp}{\mathbb P}

\DeclareMathOperator{\Aut}{Aut}
\DeclareMathOperator{\Pic}{Pic}

\DeclareMathOperator{\NS}{NS}
\DeclareMathOperator{\Hom}{Hom}

\DeclareMathOperator{\Fix}{Fix}
\DeclareMathOperator{\tors}{Tors}
\DeclareMathOperator{\Spec}{Spec}

\newcommand{\OO}{\mathcal O}
\newcommand{\epsi}{\epsilon}

\newcommand{\cF}{{\mathcal F}}

\newcommand{\cQ}{{\mathcal Q}}

\newcommand{\inv}{^{-1}}

\newcommand{\wt}{\widetilde}
\newcommand{\can}{_{\text{can}}}

\numberwithin{equation}{section}

\begin{document}
\title{Godeaux surfaces with an Enriques involution and some stable degenerations}

\author{Margarida Mendes Lopes}
\author{Rita Pardini}
\address{Rita Pardini\\Dipartimento di Matematica\\Universit\`a di Pisa \\Largo B. Pontecorvo 5\\I-56127  Pisa\\Italy}
\email{pardini@dm.unipi.it}
\address{Margarida Mendes Lopes\\Centro de An\'alise Matem\'atica, Geometria e Sistemas Din\^amicos, Departamento de Matem\'atica\\
Instituto Superior T\'ecnico, Universidade de Lisboa\\ Av. Rovisco Pais\\ 1049-001  Lisboa\\Portugal}
\email{mmlopes@math.tecnico.ulisboa.pt}

\thanks{
  The first  author is a member of the Center for Mathematical
Analysis, Geometry and Dynamical Systems of Instituto Superior T\'ecnico, Universidade de Lisboa. The second author is a  member of G.N.S.A.G.A.--I.N.d.A.M.  This research was partially supported by FCT (Portugal) through program POCTI/FEDER and   Projects  PTDC/MAT-GEO/0675/2012 and EXCL/MAT-GEO/0222/2012 and by MIUR (Italy) through  PRIN 2010-11 ``Geometria delle variet\`a algebriche".}

\date{}

\begin{abstract}
We give an explicit description of the Godeaux surfaces $S$ (minimal surfaces of general type with $K^2_S=\chi(\OO_S)=1$) that admit an involution $\sigma$ such that $S/\sigma$ is birational to an Enriques surface;  these surfaces give  a $6$-dimensional unirational irreducible  
subset 
  of  the moduli space of surfaces of general type.\\
In addition, we describe the Enriques surfaces that  are birational to the  quotient of a Godeaux surface by an involution and we show that they give a  $5$-dimensional unirational irreducible  subset of the moduli space of Enriques surfaces.\\
Finally, by   degenerating our  description we obtain  some examples of  non-normal stable Godeaux surfaces;  in particular we show that one of the families of stable Gorenstein Godeaux surfaces classified in \cite{FPR14c} consists of smoothable surfaces.

\medskip

\end{abstract}

\maketitle
\setcounter{tocdepth}{1}
\tableofcontents
\section{Introduction}

A Godeaux surface is (the canonical model of) a minimal surface of general type with $K^2_S=\chi(\OO_S)=1$. These surfaces have been intensely studied since the 1970's,  but a complete classification is still lacking. A very synthetic  summary of the state of the art is   as follows:
\begin{itemize}
\item[---] the algebraic fundamental group $\pi_1^{\rm alg}$ of a Godeaux surface is cyclic of order $\le 5$ ((\cite{Mi75}); in particular  if $S$ is a Godeaux surface then  $\pi_1^{\rm alg}$ is abelian and thus it  coincides with the torsion subgroup $\tors(S)$ of $\Pic(S)$; \item[---] the Godeaux surfaces with $\pi_1^{\rm alg}$ of order $3,4,5$ are explicitly described; to each of these possibilities for $\pi_1^{\rm alg}$ there corresponds an irreducible unirational $8$-dimensional connected component of the moduli space (\cite{reid-godeaux});
\item[---] Godeaux surfaces with $\pi_1^{\rm alg}=0$ or $\Z_2$ do exist, but little is known about the geometry of the moduli space (\cite{barlow1}, \cite{barlow2}, \cite{Lee07}, \cite{PPS13}).
\end{itemize}
 
 One of the strategies to overcome the difficulties of the classification is to restrict one's attention to a subclass of Godeaux surfaces with an extra structure, for instance those admitting an involution. This has been done  by  Keum-Lee (\cite{keum-lee00}) and  by Calabri, Ciliberto and Mendes Lopes (\cite{CaCiML07}), who described  the possibilities for the quotient surface and the fixed locus of the involution. 
 
 Here we study in detail the case when the quotient surface is birational to an Enriques surface (in this case, we call $\sigma$ an ``Enriques involution''). Since in this case $\tors(S)\cong \Z_4$  (\cite{CaCiML07}), the    universal  cover $\wt S$ of the Godeaux surface is  a complete intersection in a weighted projective space (\cite{reid-godeaux}). The involution $\sigma$ lifts to an involution $\wt \sigma$ of $\wt S$ and the action of $\wt \sigma$ on the canonical ring of $\wt S$ can be determined by means of a careful study of linear systems on the quotient Enriques surface, yielding the classification (Theorem \ref{thm:main-godeaux}). As a consequence, the locus of Godeaux surfaces with an Enriques involution is irreducible of dimension 6 (Corollary \ref{cor:Godeaux-moduli}) and the locus of Enriques surfaces that are birational to the quotient of a Godeaux surface by an involution (Enriques surfaces ``of Godeaux-quotient type'') is irreducible of dimension 5 (Corollary \ref{cor:irreducible}). 
In \S\ref{sec:enriques} we specialize a classical construction   of the special Enriques surfaces  (\cite{horikawa_periods_I},(\cite{horikawa_periods_II}) to obtain Enriques surfaces of Godeaux-quotient type: since our   construction depends on 5 parameters, by the irreducibility of the locus of Enriques surfaces of Godeaux-quotient type it gives the general Enriques surface of Godeaux-quotient type. 
\smallskip

The moduli space of (canonical models) of  surfaces of general type can be compactified by considering a larger class of surfaces, the so-called stable surfaces (cf.   \S \ref{sec:deg} for the definition). The stable  Gorenstein surfaces with $K^2=1$ (thus including the stable Gorenstein Godeaux surfaces)  are investigated in the series of recent papers \cite{FPR2014a}, \cite{FPR2014b}, \cite{FPR14c}. In \S \ref{sec:deg} we give an explicit construction of the general Godeaux surface with an Enriques involution and use it to produce stable Godeaux surfaces. In this way we produce  a normal Gorenstein degeneration with an elliptic singularity of degree 4, 
whose existence was predicted in \cite{FPR2014a},   and  we show the smoothability of one of the families of non-normal Godeaux surfaces with normalization isomorphic to $\pp^2$ (\cite{FPR2014a}, \cite{FPR14c}). In addition we give examples of stable non-normal Godeaux surfaces with Cartier index equal to 2 whose normalization is not ruled, thus showing that the main result of \cite{FPR2014a} does not hold without the Gorenstein assumption.
\smallskip

Finally, a remark on the methods: the constructions of the general Enriques surface of Godeaux-quotient type (\S \ref{sec:enriques})   and of the general Godeaux surface with and Enriques involution (\S \ref{sec:deg}) are based:
\begin{itemize}
\item[(a)] on the fact that, for  a certain involution $\tau$ of $Y$ and for a certain   double/bidouble cover $p\colon X\to Y$, $\tau$ can be lifted to an \underline{involution}  of $X$; 
\item[(b)] on the fact that  the $2$-divisibility of $p^*D$ for a certain divisor $D$ on $Y$ implies that $D$ is also $2$-divisible. 
\end{itemize}
The conditions under which   (a) and (b) above hold for a general bidouble cover are investigated  in \S 3: we believe that this section is  of independent interest. 
\bigskip

\noindent{\em Acknowledgments:} we are grateful to the editors  of   this volume for inviting us to contribute to it. We hope  that, although  the topic is   not directly related to the work of Corrado Segre, the influence of the  classical italian tradition of algebraic geometry that  pervades the paper  makes it   a suitable addition to this project.

\bigskip

{\bf Notation and conventions:} We work over the complex numbers. Following the terminology of \cite{SMMP},  a variety is called {\em demi-normal} if it satisfies condition $S_2$ of Serre and in codimension 1 it is either smooth or double crossings. If $X$ is a demi-normal projective variety, then the dualizing sheaf $\omega_X$ is divisorial;   we denote by $K_X$ a canonical divisor, that is,  a Weil divisor such that $\OO_X(K_X)\cong \omega_X$.
For a  projective  variety $X$  we denote by $\tors(X)$ the  torsion subgroup of $\Pic(X)$ and by $\Pic(X)[d]$ the subgroup consisting of the  $d$-torsion elements. 
We use $\equiv$ to denote linear equivalence of divisors and $\sim$ to denote numerical equivalence of $\Q$-divisors.

Thoughout all the paper $G$ is used to denote  the Galois group of a finite cover. 

\section{Galois covers and divisibility}
In this section we first  summarize  the theory of \cite{rita-abel} and \cite{rita-valery_non-normal} for covers with Galois group $\Z_2$ and $\Z_2^2$; the need to cover  also the case of non-normal covers arises because in   \S \ref{sec:deg} we consider  stable Godeaux surfaces.
 
Then  we present some general results on liftability of automorphisms to double and bidouble covers that are needed in the rest of the paper.  Although these results are probably known to experts, to our knowledge they have not been written down elsewhere and we believe that they are of independent interest. 

\subsection{Double and bidouble covers}\label{ssec:bidouble}
Let $G$ be a finite group. A $G$-cover  is a finite  map of  algebraic varieties  $f\colon X\to Y$  that is  the quotient map for a generically faithful $G$-action, namely such that for every component $Y_i$ of $Y$ the $G$-action on the restricted cover $X\times_Y Y_i\to Y_i$ is faithful. The cover is {\em abelian} if $G$ is an abelian group: for  the general theory of abelian covers   we refer the reader  to \cite{rita-abel} for the case  $X$  normal  and $Y$ smooth and to \cite{rita-valery_non-normal} for a more general treatment. 

Here we are mainly interested  in the case $G\cong \Z_2$ (``double covers'') and $G\cong \Z_2^2$ (``bidouble covers''); for simplicity, we assume throughout that $H^0(Y,\OO_Y)=\C$.
 
 Assume first that $f\colon X\to Y$ is an abelian cover with group $G$ such that   $X$ is normal   and $Y$ is smooth. Then  $f$  is flat and  the branch locus  is a   divisor; we denote by  $B$ the branch divisor with  reduced structure. For $G=\Z_2$,  we have $f_*\OO_X=\OO_Y\oplus L\inv$, where $L$ is a line bundle, $G$ acts on $L\inv $ as multiplication by $-1$ and the multiplication map $L\inv \otimes L^{\inv}\to \OO_Y$ induces an isomorphism $L^{\otimes 2}\cong \OO_Y(B)$. The pair $(L, B)$ is called the   {\em building  data} of the double cover  and it determines $f\colon X\to Y$ uniquely up to isomorphism of covers, since we assume $H^0(\OO_Y)=\C$. We say for short that $f\colon X\to Y$ is the double cover given by the equivalence relation $2L\equiv B$. 
 
 One can reverse this  construction: given building data $(L,B)$, i.e. given an effective divisor $B$ and a line bundle $L$ satisfying the relation $2L\equiv B$, one can choose an isomorphism $\phi\colon L^{\otimes 2}\to  \OO_Y(B)$,  use it to define an associative multiplication on $\OO_Y\oplus L\inv$, set  $X:=\Spec(\OO_Y\oplus L\inv)$ and take $f$ to be  the natural map $X\to Y$. This construction makes sense more generally  for any effective Cartier divisor $B$ (not necessarily reduced) and  line bundle $L$ such that $2L\equiv B$ on an arbitrary variety $Y$. The flat  double cover $f\colon X\to Y$  is   called  the {\em standard cover} associated with  $(L,B)$; it is not hard to show that every  flat double cover is  obtained this way, i.e., it is standard.
 
 The situation is similar for   bidouble covers. We start again by considering the case $X$ normal and $Y$ smooth.   We  write $\chi_1, \chi_2, \chi_3$ for the three non-trivial characters of $G\cong \Z_2^2$ and denote by $g_i\in G$ the generator of $\ker \chi_i$. The branch divisor $B$ decomposes as $B=B_1+B_2+B_3$, where $B_i$ is the image of the divisorial part of the fixed locus of $g_i$ and 
we have a splitting  $f_*\OO_X=\OO_Y\oplus L_1\inv \oplus L_2\inv \oplus L_3\inv$, where $G$ acts on  $L_i\inv$ as multiplication by  the character $\chi_i$. As in the case of double covers, the multiplication in $f_*\OO_X$  induces isomorphisms, and therefore  equivalence relations:
\begin{equation}\label{eq:fundrel}
2L_i\equiv B_j+B_k,\quad L_i+L_j\equiv L_k+B_k,
\end{equation} 
where $(i,j,k)$ is a permutation of $(1,2,3)$.  Again,    $(L_i, B_i)$, $i=1,2,3$, are called the building data of the bidouble  cover and determine $f\colon X\to Y$ up to isomorphism of $\Z_2^2$-covers. It is easy to see that \eqref{eq:fundrel} is equivalent to the smaller set of equations:
\begin{equation}\label{eq:redrel}
2L_1\equiv B_2+B_3, \quad 2L_2\equiv B_1+B_3,\quad L_3\equiv L_1+L_2-B_3,
\end{equation}
and in particular $L_3$ can be recovered from the remaining data. We call $(L_1,L_2, B_1,B_2,B_3)$ the {\em reduced building data} and we say
 for short that the cover is given by the relations $2L_1\equiv B_2+B_3$, $2L_2\equiv B_1+B_3$. 

As in the case of double covers, we can perform  the  reverse construction  in greater generality, starting with  line bundles $L_1, L_2$ and  effective Cartier divisors satisfying \eqref{eq:redrel}, and  obtain a   {\em standard bidouble cover} of an arbitrary variety $Y$. 
Again, the building data determine the standard cover uniquely up to isomorphism of bidouble covers, since we assume $H^0(\OO_Y)=\C$. 
We set $B=B_1+B_2+B_3$; observe that $B$ may be non-reduced. 
We recall the following:
\begin{prop}[\cite{rita-valery_non-normal}, Cor.~1.10]
Let $f\colon X\to Y$ be a  a double or bidouble cover  with $Y$ smooth and $X$ demi-normal. 
 Then   $f$ is a standard cover  and every component of $B$ has multiplicity at most 2.
\end{prop}

\subsection{Lifting automorphisms to double and  bidouble covers}\label{ssec:lift}
We discuss in detail the case of bidouble covers; the case of double covers can be treated by similar, but simpler, arguments. 

Let $Y$ be a variety with $H^0(Y,\OO_Y)=\C$, let  $f\colon X\to Y$ be a  standard  bidouble cover  given by relations $2L_1\equiv B_2+B_3$ and $2L_2\equiv B_1+B_3$  and denote by $G\cong \Z_2^2$ the Galois group of $f$. Let   $\rho\in \Aut(Y)$ be an automorphism such that one of the following holds:
\begin{enumerate}[(a)]
\item $\rho^*B_i=B_i$, $i=1,2,3$,  and $\rho^*L_j\equiv L_j$, $j=1,2$
\item $\rho^*B_1=B_2$, $\rho^*B_2=B_1$, $\rho^*B_3=B_3$, $\rho^*L_1\equiv L_2$, $\rho^*L_2\equiv L_1$.
\end{enumerate}
In either case, the automorphism $\rho$ lifts to an automorphism $\wt \rho$ of $X$. Indeed, consider the following cartesian diagram:
\begin{equation}
\begin{CD}
X'@>\rho'>>X\\
@V{f'} VV @VV f V\\
Y@> \rho >>Y.
\end{CD}
\end{equation}
In case (a), $f'$ is a  standard bidouble cover given by the same building data  as $f$, hence it is isomorphic to $f$ via an isomorphism  compatible with the  action of $G\cong \Z_2^2$ and $\wt \rho$ is obtained by composing such an isomorphism with $\rho'$; 
in case (b)  we modify the $G$-action on $X'$ by composing with the automorphism of $G$ that switches  $g_1$ and $g_2$ and argue as in case (a). 

Let $\wt G$ be the subgroup of $\Aut(X)$ generated by $G$ and by $\wt \rho$. Then there is a short exact sequence of groups:
\[1\to G\to \wt G \to <\rho>\to 1.\]
The group $\wt G$ is abelian in case (a), since $\wt \rho$ preserves the decomposition of $f_*\OO_X$ into $G$-eigensheaves,  and it is non abelian in case (b); in particular,  if $\rho^2=1$ then ${\wt \rho}\,^4=1$ and, by the classification of groups of order 8,  $\wt G$ is isomorphic either to $\Z_2^3$ or $\Z_2\times \Z_4$ in case (a) and to the dihedral group $D_4$ in case (b). 

In the case of  double covers one assumes that $\rho^*B=B$ and  $\rho^*L\equiv L$: in this case $\wt \rho$ commutes with the action of $G\cong \Z_2$ and   the group $\wt G$ is  isomorphic to $\Z_2\times \Z_d$ or to $\Z_{2d}$.

\subsection{Divisibility}
Recall  that a Cartier divisor or line bundle on a projective variety is said to be  {\em even} if its class is divisible by $2$ in $\Pic(X)$.
\begin{lem}\label{lem:div} Let $f\colon X\to Y$ be a cyclic \'etale cover  of projective varieties and let   $K$ be the kernel of $f^*\colon \Pic(Y)\to \Pic(X)$.
Let $D$ be a Cartier divisor on $Y$ such that $f^*D$ is even. 

If $\Pic(X)[2]=0$, then the class of $D$ is divisible by 2 in $\Pic(Y)/K$.
\end{lem}
\begin{proof} Let $\wt M\in \Pic(X)$ be a line bundle such that $2\wt M\equiv f^*D$. Denote by $g$ a generator of the Galois group $G$  of $f$; since $D$ is $g$-invariant, we have $2g^*\wt M\equiv f^*D\equiv 2\wt M$. Since $\Pic(X)[2]=0$ it follows that the lines bundles $\wt M$ and $g^*\wt M$ are isomorphic and therefore $\wt M$ admits a $G$-linearization ($G$ is cyclic). Since $f$ is \'etale, $\wt M$ descends to a line bundle $M$ on $Y$. One has $f^*(2 M-D)\equiv 0$, hence $D=2 M$ in $\Pic(Y)/K$.
\end{proof}
\begin{lem}\label{lem:div-geo} Let $f\colon X\to Y$ be a cyclic \'etale cover  of degree $d$ of projective  varieties and let $D$ be  an  effective Cartier  divisor on $Y$ such that $f^*D$ is even.   Assume that $\Pic(X)[2]=0$ and denote by $h\colon Z\to X$ the flat double cover branched on $f^*D$. Then:
\begin{enumerate}
\item the composite map $f\circ h\colon Z\to Y$ is a Galois cover with Galois group $\wt G$ isomorphic  to $\Z_{2d}$ or to $\Z_2\times \Z_d$;
\item $D$ is even iff   $\wt G$ is isomorphic to $\Z_2\times \Z_d$.
\end{enumerate}
\end{lem}
Notice that Lemma \ref{lem:div-geo} is interesting only if $d$ is even. Indeed, if $d$ is odd then $\wt G\cong\Z_{2d}\cong \Z_2\times \Z_d$ is cyclic and statement (ii) just says that $D$ is even, as we already know by Lemma \ref{lem:div}.

\begin{proof}
(i) Let $\wt L\in \Pic(X)$ be the only element such that $2\wt L\equiv f^*D$. Let $g$ be a generator of the Galois group $G$  of $f$; by construction $f^*D$ is $G$-invariant, hence arguing as in the proof of Lemma \ref{lem:div} one sees that $\wt L$ is also $G$-invariant. Therefore  by the discussion of  \S \ref{ssec:lift} it is possible to lift $g$ to an automorphism $\wt g$ of $Z$ and the subgroup $\wt G$ of  $\Aut(Z)$ generated by $\wt g$ and   by  the involution $\iota$ associated with $h$ is isomorphic to $\Z_{2d}$ or $\Z_2\times \Z_d$.  The former case occurs iff $\wt G$ is generated by $\wt g$ or by $\wt g\iota$. Clearly, $\wt G$ is the Galois group of $f\circ h$. \medskip

(ii) Assume that $\wt G\cong \Z_2\times \Z_d$ and let $\wt g$ be an element of order $d$ that lifts $g$: then  $\wt g $ acts freely on $Z$ by construction and $Z/\wt g\to Y$ is a flat  double cover. Since $Z\to Z/\wt g$ is \'etale, it is easy to see that $Z/\wt g\to Y$ is standard with building data $(L,D)$, for some  $L\in \Pic(Y)$,   hence $D$ is even. 
Conversely, assume that $D$ is even and let $ L\in \Pic(Y)$ be  such that $2 L\equiv D$.  We have $f^* L=\wt L$ since $\Pic(X)[2]=0$ and therefore $Z\to Y$ is the fiber product of  $f\colon X\to Y$ and  of the double cover given by the relation $2 L\equiv D$  and has Galois group isomorphic to $\Z_2\times \Z_d$.
\end{proof}
Let $X$ be  a  surface and let $p_1,\dots,p_k\in X$ be $A_1$ singularities (``nodes''). We say that $p_1, \dots ,p_k$ is an {\em even set of nodes} of $X$  if there exists a double cover of $X$ branched precisely on $p_1,\dots, p_k$. Denote by $X'\to X$ the minimal resolution of the singularities $p_1,\dots p_k$  and by $C_i$ the exceptional curve over $p_i$; $C_i$ is a {\em nodal} curve, i.e., it is smooth rational and $C_i^2=-2$. The set $\{p_1,\dots p_k\}$ is  even if and only if $C_1+\dots +C_k$ is an even divisor of $X'$. By using the adjunction formula on $X'$ it is easy to check that an even set of nodes has cardinality divisible by 4.
\begin{lem} \label{lem:even}
Let $Y$ be a  smooth projective surface, let $B_1$, $B_2$ be  even  curves of $Y$ meeting transversely at  smooth points $q_1,\dots q_k$ of $Y$. 

If   $f\colon X\to Y$ is  a flat double cover branched on $B:=B_1+B_2$, then  the points $p_1,\dots p_k$ lying above $q_1,\dots ,q_k$ are an even set of nodes of $X$.
\end{lem} 
\begin{proof}  The fact that $p_1,\dots p_k$ are nodes of $X$ can be checked easily by a local computation.
Let $L_3\in \Pic(X)$ be such that $f_*\OO_X=\OO_Y\oplus L_3\inv$, so that $f$ is given by the relation $2L_3\equiv B$. Choose  $L_1\in \Pic(X)$ with $2L_1\equiv B_2$ and set $L_2:=L_3-L_1$.  As explained in \S \ref{ssec:bidouble}, the relations $2L_1\equiv B_2$ and $2L_2\equiv B_1$ determine a standard  bidouble cover $h\colon Z\to Y$ (we take $B_3=0$). For $i=1,2$ denote  by $g_i\in G\cong \Z_2^2$ the element that fixes $h\inv B_i$ pointwise and set $g_3=g_1+g_2$. Then $Z/g_3$ is isomorphic to $X$ and the quotient map $Z\to Z/g_3$ is a double cover branched precisely on $p_1,\dots ,p_k$.
\end{proof}

\section{Godeaux surfaces with an Enriques involution}
In this section we study the following situation:
\begin{itemize}
\item $S$ is a numerical Godeaux surface, i.e.,  a smooth minimal surface of general type with $K^2_S=1$ and $p_g(S)=q(S)=0$
\item $\sigma\in \Aut(S)$ is an involution such that $\Sigma:=S/\sigma$ is birational to an Enriques surface.
\end{itemize}
We call the involution $\sigma$ an {\em Enriques involution}. 
Godeaux surfaces with an involution have been studied in \cite{keum-lee00} and    in  \cite{CaCiML07}; in particular, in \cite{CaCiML07} it is proven that a Godeaux surface $S$   with an Enriques involution has $\tors(S)\cong \Z_4$.  In addition, the possible automorphism groups of numerical Godeaux surfaces with torsion of order $\ge 3$ have been listed in \cite{maggiolo}, but without analyzing  the quotient surfaces. 

We recall  the following example  \cite[Ex.~4.3]{keum-lee00}: 

\begin{ex}\label{ex:keum-lee}
Let $S$ be a Godeaux surface with $\tors(S)\cong \Z_4$ and let $\wt S\to S$ be the universal  cover, i.e. the degree 4 cyclic cover given by $\tors(S)$. By \cite{reid-godeaux}, the  minimal model $\wt S\can$ of $\wt S$ is canonically embedded  in $\pp(1,1,1,2,2)$, with coordinates $x_1,x_2,x_3, y_1,y_3$,  as the zero locus of  two homogeneous equations $q_0$ and $q_2$ of degree 4. 

The equation $q_0$ involves the monomials:
\[ x_1^4, x_2^4, x_3^4, x_1^2x_3^2, x_1x_3x_2^2, x_1x_2y_1, x_2x_3y_3, y_1y_3,\]
 and $q_2$ involves the monomials:
\[x_1^2x_2^2, x_2^2x_3^2, x_1^3x_3, x_1x_3^3, x_1x_2y_3, x_2x_3y_1, y_1^2, y_3^2.\]
 We denote by $G\cong \Z_4$  the Galois group of $\wt S\to S$: the group $G$ acts freely also on $\wt S\can$ and the quotient surface is the canonical model $S\can$ of $S$. The action of $G$ extends to the ambient  $\pp(1,1,1,2,2)$ and there is a generator $g\in G$ that acts   by $(x_1,x_2, x_3,y_1,y_3)\mapsto (ix_1,-x_2,-ix_3,iy_1,-iy_3)$. 
 
 Now we define an involution $\wt\sigma$ of $\pp(1,1,1,2,2)$ by $(x_1,x_2, x_3,y_1,y_3)\mapsto (-x_1,x_2,-x_3,y_1,y_3)$; the involution $\wt \sigma$ commutes with $g$.   We assume from now on  that the polynomial $q_0$ does not involve $x_1x_2y_1$, $x_2x_3y_3$ and the polynomial $q_2$ does not involve $x_2x_3y_1$, $x_1x_2y_3$, so that $q_0$ and $q_2$ are invariant under 
 $\wt \sigma $. Hence $\wt \sigma$  acts on   $\wt S\can$ and descends to an  involution $\sigma$ of $S\can$ and of its minimal resolution $S$.
  
 The divisorial part $R$ of the fixed locus $ \sigma$ on $S\can$ is the paracanonical curve defined by $x_2=0$, hence  it is a connected curve of genus 2; if $S\can$ is smooth then $R$ is also smooth, and by Cor.~4.8 and Prop.~7.10 of \cite{CaCiML07} it follows that $\sigma$ is an Enriques involution. Since the quotient of a  smooth surface by an involution has canonical singularities, it follows that for every smooth $S\can$ as above the involution $\sigma$ of $S$ is an Enriques involution.
Using Bertini's theorem, it is not difficult to see that if  $q_0$ anq $q_2$ are general  the surface $S\can=S$ is smooth. 
 \end{ex}
 In this section we characterize the quotient surface $S/\sigma$ and, exploiting this characterization,  we prove the following classification results:
 \begin{thm}\label{thm:main-godeaux}
 Let $S$ be a Godeaux surface and let $\sigma\in \Aut(S)$ be an Enriques involution.
 
 Then $S$ is as in Example \ref{ex:keum-lee}.
  \end{thm}
  The surfaces in Example \ref{ex:keum-lee}  correspond to case $R_1$ of Table 2 of \cite{maggiolo}, hence they form an  irreducible unirational subset  of dimension 6 of the moduli space of Godeaux surfaces with torsion of order 4. Hence Theorem \ref{thm:main-godeaux} yields immediately:
  \begin{cor}\label{cor:Godeaux-moduli}
 The Godeaux surfaces with an Enriques involution give an   irreducible  unirational    subset $\mathcal{GE}$ of dimension 6 of the moduli space of Godeaux surfaces with torsion of order 4. 
 \end{cor}

 A possible strategy for proving  Theorem \ref{thm:main-godeaux} would be  to use the description  given in \cite{maggiolo} of the Godeaux surfaces with torsion of order 4 that admit an involution and decide which involutions are Enriques by looking at the fixed locus, as we have done  in Example \ref{ex:keum-lee}.   However we prefer to use a more conceptual approach, based on a detailed study  of linear systems  on  the quotient Enriques surfaces, that gives also a description of the family of such  Enriques surfaces (cf. \S \ref{sec:enriques}).
   \bigskip

The rest of the section is devoted to proving  Theorem \ref{thm:main-godeaux}; we start by fixing some notation.

We denote by $\pi\colon S\to \Sigma$ the quotient map; by \cite[Prop.~4.5]{CaCiML07}, the bicanonical map of $S$ is composed with $\sigma$  and $\Fix(\sigma)$ consists of a smooth curve $R$ and of  5 isolated fixed points $p_1, \dots p_5$. We set $q_i=\pi(p_i)$, $i=1,\dots ,5$ and  $B:=\pi(R)$.  There is a commutative diagram
\begin{equation}
\begin{CD}
V@>{\epsilon}>>S\\
@V{\wt \pi} VV @VV\pi V\\
W@> \eta >>\Sigma
\end{CD}
\end{equation}
where $\epsi$ is  the blow up of $S$ at $p_1,\dots p_5$, $\eta$ is the minimal resolution of $\Sigma$  and $\wt \pi$ is a flat double cover.  
For $i=1,\dots,5$ we denote  by $C_i$ the exceptional curve over $q_i$; the $C_i$ are nodal curves, that is, they are smooth rational and $C_i^2=-2$. By \cite[Prop.~3.9]{CaCiML07} and Lemma 4.11, ibidem, there exists a birational morphism $f\colon W\to Y$ such that:
\begin{itemize}
\item $Y$ is a smooth Enriques surface
\item the exceptional locus  of $f$ is disjoint from the $C_i$
\item  there is a flat double cover $p\colon X\to Y$ fitting in the commutative diagram:
\begin{equation}\label{diag:grande}
\begin{CD}
X@<g<<V@>{\epsilon}>>S\\
@VpVV @V{\wt \pi} VV @VV\pi V\\
Y@<f<<W@> \eta >>\Sigma
\end{CD}
\end{equation}
where $X$ has canonical singularities and $g$ is the minimal resolution.
\end{itemize}
Also, we abuse notation and we denote by the same letter a  curve in $V$, resp. $W$,  and its  image in  $X$, resp. $Y$. This should not be confusing for the reader, since we will mostly work with the cover $p\colon X\to Y$ and forget about $\wt \pi\colon V\to W$.
The branch curve $B\subset Y$ has at most negligible singularities and it is disjoint from $C_1,\dots C_5$;  the flat cover $p$ is   given by the linear equivalence $2L\equiv B+C_1+\dots+C_5$.
 For $i=1,\dots 5$, the surface $X$ is smooth above the curve $C_i$ and $p^*C_i=2\Gamma_i$, with $\Gamma_i$ a $-1$-curve. By contracting  $\Gamma_1, \dots \Gamma_5\subset X$, one obtains an intermediate   object between  the minimal surface $S$ and its canonical  model $S\can$; in particular   $p^*B$ is   the pull back of   $2K_{S\can}$, hence $B$   is nef and  $B^2=2$.  Since  $h^i(B)=h^i(K_Y+(K_Y+B))$,  by Kawamata-Viehweg vanishing we have $h^i(B)=0$ for $i>0$, so $h^0(B)=2$.
 We have $L^2=-2$, hence $\chi(L)=0$. Since $h^i(L)=0$ for $i>0$ by Kawamata-Viehweg vanishing, we have $h^0(L)=0$ as well. 

Recall (cf. \cite{cossec-dolgachev}) that an {\em elliptic half-pencil} of an Enriques surface $Y$ is an effective divisor $E$ such that $|2E|$ is a free pencil of elliptic curves of $Y$. One has: 

\begin{prop}\label{prop:parity} In the above setting, up to reordering  $C_1,\dots , C_5$,  we have:
\begin{enumerate}
\item there exists an elliptic half-pencil $E$ of $Y$  such that $B\in |2E+C_5+K_Y|$;
\item the divisor $K_Y+C_1+\dots +C_4$ is  divisible by 2 in $\Pic(Y)$.  
\end{enumerate}
\end{prop}
\begin{proof} (i) Let $D\in |B|$ be general.
By  \cite[Prop.~5.1]{CaCiML07}, $D$  is irreducible; since $D^2=2$, by Bertini's theorem  it follows that  $D$ is smooth. 
Consider the system $|M|=|2B|$: the (set-theoretic) base locus of $|M|$ is contained in the (set-theoretic) base locus of $|B|$, which consists of $1$ or $2$ points.  The restriction  sequence
$0\to H^0(B)\to H^0(M)\to H^0(2K_D)\to 0$ is exact, since $H^1(B)=0$; it follows that $|M|$ is free and, in the terminology of \cite{cossec-dolgachev}, it is a {\em superelliptic} system. By \cite[Thm.~4.7.1]{cossec-dolgachev}, $M=2B'$, where   there are two possibilities for $B'$:
\begin{enumerate}[(a)]
\item there exists  elliptic half-pencils $E_1$, $E_2$ such that $E_1E_2=1$ and $B'=E_1+E_2$
\item there exists an elliptic half-pencil $E$ and a nodal curve $Z$ such that $EZ=1$ and $B'=2E+Z$
\end{enumerate}
Since $2B=2B'=M$, we either have $B=B'$ or $B=B'+K_Y$, and in either case  $B$ and $B'$ are numerically equivalent. 
If   case (a) occurs,  then $(E_1+E_2)B=2$ and $(E_1+E_2)C_i=0$ for $i=1,\dots, 5$. Since $|2E_i|$ is a free pencil  for $i=1,2$ and $B^2>0$,  it follows that $E_iB=1$ and $E_iC_j=0$ for $j=1,\dots, 5$. So we have $E_i(2L)=E_i(B+C_1+\dots+C_5)=1$, a contradiction.
So case (b) occurs. We claim that $Z$ is one of the $C_i$. Assume by contradiction that this is not the case: then $(2E+Z)C_i=0$ implies that $Z$ is disjoint from the $C_i$. The divisor $C_1+\dots +C_5+Z\sim 2L-2E$ has self-intersection $-12$, hence $(L-E)^2=-3$, contradicting the fact that the intersection form on  $\NS(Y)$ is even.

So $Z$ is equal to, say, $C_5$, and we have $B=2E+C_5+K_Y$, since $|2E+C_5|$ has $C_5$ as a fixed component while $|B|$ is an irreducible system. 
\medskip

(ii) follows immediately by (i).
\end{proof}

\begin{lem}\label{lem:torsion}
 Let $S$ be a Godeaux surface with an involution $\sigma$ of Enriques type. 
Then $\tors(S)$ is cyclic of order 4 and $\sigma$ acts as the identity on $\tors(S)$.
\end{lem}
\begin{proof} That $\tors(S)$ is cyclic of order 4 is proven in \cite[Prop.~5.3]{CaCiML07}. Here we describe explicitly $\tors(S)$. 
Since smooth blow ups do not change the torsion, we may replace $S$ by $X$. Of course the element of order 2 is $p^*K_Y$. By Proposition \ref{prop:parity} there is $N\in \Pic(Y)$ such that $2N\equiv C_1+\dots +C_4+K_Y$; pulling back to $X$  we obtain $2p^*N\equiv 2(\Gamma_1+\dots+\Gamma_4)+p^*K_Y$, hence $p^*N-(\Gamma_1+\dots +\Gamma_4)$ is a torsion element of order 4 and it is clearly $\sigma$-invariant. 
\end{proof} 

\begin{lem} \label{lem:lift} Let $S$ be a Godeaux surface with an involution $\sigma$ of Enriques type,  let $c\colon\wt S\to S$ be the  canonical cover and let $G=\Hom(\tors(S), \C^*)$ be the Galois group of $c$. Then there is an involution $\wt \sigma $ of $\wt S$ that lifts $\sigma$ and commutes with $G$.
\end{lem}
\begin{proof} Since the canonical cover is intrinsically associated with $S$, $\sigma$ can be lifted   to an automorphism $h$ of $\wt S$, so the point is to show that $h$ can be taken to be an involution that commutes with $G$. We have 
$$\wt S=\Spec(\oplus_{\eta\in \tors(S)}\eta),$$
  and by Lemma \ref{lem:torsion}  the action of  $G$  on  $\oplus_{\eta\in \tors(S)}\eta$ preserves  the summands. Thus  $h$ commutes with $G$.  Denote by $\wt G$ the subgroup of $\Aut(\wt S)$ generated by $h$ and $G$: it is an abelian group of order 8 with a cyclic subgroup of order $4$, hence it is either isomorphic to $\Z_4\times \Z_2$ or to $\Z_8$. To prove the lemma we have to exclude the latter possibility. Assume for contradiction that $\wt G$ is cyclic of order 8: 
  then $h$ generates $\wt G$.  So  in particular $h$ acts freely on $\wt S$, because $G$  does so. It follows that the group $\wt G$ acts freely on $\wt S$, which is impossible, for instance because $K^2_{\wt S}=4$ is not divisible by 8. 
\end{proof}

By definition, the canonical ring $R(\wt S)$ coincides  with the paracanonical ring of $S$:
\[\oplus_{m\in \mathbb N, \ \eta\in \tors(S)} H^0(mK_S+\eta).\]

 There are two possible choices   of $\wt \sigma$ as in Lemma \ref{lem:lift};  each of these choices induces a $\sigma$-linearization of the pluricanonical bundles $mK_S+\eta$ compatible with the multiplicative structure of $R(\wt S)$ and a $\Z_2$-action  on $H^0(mK_S+\eta)$ that lifts $\sigma$. So each vector space $H^0(mK_S+\eta)$ splits as a sum of two eigenspaces (corresponding to $\pm 1$), whose dimensions we call the {\em $\sigma$-type} of $mK_S+\eta$. 

We determine the $\sigma$-type in some cases:

\begin{lem}\label{lem:types}
Let $S$ be a Godeaux surface and $\sigma \in \Aut(S)$ an Enriques involution. Denoting by  $1\in\tors(S)$ a generator, 
the $\sigma$-type of $mK_S+i$, for  $m=1,2,4$ and $i\in \tors(S)$ is  shown  in row $m$, column $i$ of Table \ref{table:types}.
\end{lem}

\begin{table}[h]\caption{$\sigma$-types of  $mK_S+i$}\label{table:types}
	\begin{tabular}{|c|c|c|c|c|}
\hline
m $\backslash$ i&0&1&2&3\\
\hline
1& \{0,0\}& \{1,0\}& \{1,0\}& \{1,0\}\\
\hline
2&\{2,0\}& \{1,1\}& \{2,0\}& \{1,1\}\\
\hline
4&\{5,2\}&\{4,3\}&\{5,2\}&\{4,3\}\\
\hline

\end{tabular}
\end{table}

\begin{proof}
We may replace $S$ by $X$, since this does not affect the $\sigma$-type. We recall the Hurwitz formula $K_X=p^*(K_Y+L)$, where as usual  $L$ is the line bundle such that $p_*\OO_X=\OO_Y\oplus L\inv$; in addition, by Lemma \ref{lem:torsion} and its proof, there is a line bundle $N\in \Pic(Y)$ such that  $2N\equiv K_Y+C_1+\dots +C_4$ and our  chosen generator $1\in \tors(X)\cong\tors(S)$ is equal to $p^*N-(\Gamma_1+\dots +\Gamma_4)$.
So we have :
\begin{align}\label{eq:mK}
mK_X& \equiv  p^*(mK_Y+mL)\\
mK_X+1& \equiv p^*(mK_Y+mL+N)-(\Gamma_1+\dots +\Gamma_4)\nonumber \\
mK_X+2& \equiv p^*((m+1)K_Y+mL)\nonumber \\
mK_X+3 &\equiv p^*((m+1)K_Y+mL+N)-(\Gamma_1+\dots +\Gamma_4).\nonumber 
\end{align}
Recall also that $h^0(K_X+i)=1$ for $i\ne 0$ and $h^0(mK_X+i)=1+\frac {m(m-1)}{2}$ for $m\ge 2$ and for every $i\in \tors(X)$. Using these remarks, the  projection formulae for double covers and Kawamata-Viehweg vanishing,  it is not hard  to obtain Table \ref{table:types}.

As an example, consider $2K_X+1$:  using \eqref{eq:mK} and the relation $2L\equiv B+C_1+\dots +C_5$,  gives $2K_X+1=p^*(B+N)+\Gamma_1+\dots +\Gamma_4+2\Gamma_5$. Since  for $m>0$ and for every $i$ the fixed part of $|mK_X+i|$ contains $m(\Gamma_1+\dots +\Gamma_5)$, we have  $2=h^0(2K_X+1)=h^0(p^*(B+N))$. The projection formula for double covers gives the following  decomposition in $\Z_2$-eigenspaces:
\[H^0(p^*(B+N))=H^0(B+N)\oplus H^0(B+N-L).\]
We have $B+N\sim B+\frac 1 2(C_1+\dots+C_4)$: since $B$ is nef and big, we may apply Kawamata-Viehweg vanishing and we obtain $h^0(B+N)=\chi(B+N)=1$, and thus $2K_X+1$ has $\sigma$-type $\{1,1\}$.

\end{proof}

\begin{proof}[Conclusion of the proof of Theorem \ref{thm:main-godeaux}]
We follow the steps of Reid's description of the paracanonical ring $R(S)$ taking into account also the action of  the cyclic group $ G$ (of order 4). So in degree 1 we have  generators $x_i\in H^0(K_S+i)$, $i=1,2,3$ and in degree 2 we have two more generators $y_j\in H^0(2K_S+j)$, for $j=1,3$ and the element $g\in G$  acts on these generators as  in Example \ref{ex:keum-lee}. In addition, we may assume that all these generators are eigenvectors of $\wt \sigma$, since $\wt \sigma $ and $g$ commute. Finally, up to replacing $\wt \sigma$ by $\wt\sigma g^2$,  we may assume that $y_1$ is $\wt \sigma$ invariant. 
The space $H^0(2K_S)$ is generated by $x_2^2$ and $x_1x_3$: since by Lemma \ref{lem:types} the $\sigma$-type of $2K_S$ is $\{2,0\}$, it follows that $x_1$ and $x_3$ are eigenvectors of $\wt \sigma$ for the same eigenvalue. The space $H^0(2K_S+1)$ is generated by $x_2x_3$ and $y_1$ and has type $\{1,1\}$. It follows that $x_2$ and $x_3$ have opposite eigenvalues. Similarly, looking at $H^0(2K_S+3)$ we conclude that $y_3$ is also $\wt \sigma$-invariant. 
So $\wt \sigma$ has the form $(x_1, x_2, x_3,y_1,y_3)\mapsto (\pm x_1,\mp x_2, \pm x_3, y_1, y_3)$. 

Now look at $H^0(4K_S)$: the two eigenspaces are spanned by
\begin{equation}\label{eq:q0}
x_1^4, x_2^4, x_3^4, x_1^2x_3^2, x_1x_3x_2^2,  y_1y_3
\end{equation}
and by 
\[x_1x_2y_1, x_2x_3y_3.\]
Since by Lemma \ref{lem:types} the $\sigma$-type of $4K_S$ is $\{5,2\}$, there is a linear relation $q_0$ involving the monomials \eqref{eq:q0}.
The same argument shows the existence of a relation $q_2$ between the monomials:
\[x_1^2x_2^2, x_2^2x_3^2, x_1^3x_3, x_1x_3^3,  y_1^2, y_3^2.\]

Finally, we observe that  the map $(x_1,x_2, x_3, y_1,y_3)\mapsto (-x_1,-x_2,-x_3,y_1,y_3)$ induces the identity on $\pp(1,1,1,2,2)$,  so $\wt \sigma$ acts  on $\pp(1,1,1,2,2)$ as in Example \ref{ex:keum-lee}.

\end{proof}

\section{Enriques surfaces of Godeaux-quotient type}\label{sec:enriques}

Here we apply the results of the previous section to describe the Enriques surfaces that  are (birational)  quotients of a Godeaux surface by an involution. 

We consider Enriques surfaces $Y$  such that $Y$ contains an elliptic half-pencil $E$ and nodal curves $C_1,\dots C_5$ such that:
\begin{itemize}
\item $EC_5=1$, $EC_1=\dots =EC_4=0$
\item $C_1+\dots+C_4+K_Y$ is divisible by 2 in $\Pic(Y)$. 
\end{itemize}

We call a surface $Y$ as above an Enriques surface of  {\em Godeaux-quotient type}. 
Proposition \ref{prop:parity} has a converse:
\begin{prop}\label{prop:converse}
In the above setting:
\begin{enumerate}
\item the system $|2E+C_5+K_Y|$ is an irreducible pencil;
\item let $B\in |2E+C_5+K_Y|$ be a curve disjoint from $C_1,\dots C_5$; then there exists a double cover $X\to Y$ branched on $B+C_1+\dots +C_5$
and the minimal model of $X$ is a Godeaux surface with an involution of Enriques type. 
\end{enumerate}
\end{prop}
\begin{proof}

The first assertion follows from the Riemann-Roch theorem and  \cite[Proposition 3.1.5]{cossec-dolgachev}.  For (ii)  notice that $B+C_1+\dots +C_5$  is even.  Let $X\to Y$ be a double cover branched on $B+C_1+\dots +C_5$. Standard double cover calculations (cf. for example \cite[Proposition 2.2]{marg-rita_K23})  yield the result.

\end{proof}

As a direct consequence of Propositions \ref{prop:converse}, \ref{prop:parity} and Corollary \ref{cor:Godeaux-moduli},  we have:
\begin{cor}\label{cor:irreducible}
The Enriques surfaces of Godeaux-quotient type are an irreducible unirational   subset of dimension 5 of the moduli space of Enriques surfaces. 
\end{cor}

We  now give an explicit construction of Enriques surfaces of Godeaux quotient type. 

\begin{ex}\label{ex:main-enriques}
Consider the quadric cone $\mathcal Q\subset \pp^3$ defined by $y_0^2-y_1y_2=0$ and the involution $\tau$ of $\mathcal Q$ defined by $[y_0,y_1,y_2,y_3]\mapsto [y_0,-y_1,-y_2, y_3]$. The linear system $|M|$   spanned by the  invariant quadrics $y_0^2, y_1^2, y_2^2, y_3^2, y_0y_3$ embeds the quotient surface $\mathcal Q/\tau$ in $\pp^4$ as a quartic surface $\mathcal D$ defined by $x_0^2-x_1x_2=x_0x_3-x_4^2=0$.  The surface $\mathcal D$ ($\mathcal D_1'$
 in the notation of \cite[Ch.0,~\S 4]{cossec-dolgachev})  has two singular points of type $A_1$ at the points $P_1=[0,1,0,0,0]$ and $P_2=[0,0,1,0,0]$ (the ``simple vertices'') and a singularity of type $A_3$ at the point $P_0=[0,0,0,1,0]$ (the ``$A_3$-vertex''). 

An Enriques surfaces is called {\em special} if it contains a nodal curve $C$ and an elliptic half-pencil $E$ with $EC=1$. All the special Enriques surfaces can be constructed as follows (cf. \cite{horikawa_periods_I} and \cite{horikawa_periods_II}).

Take an element $B_0$ in the  linear system  of  $\tau$-invariant  quartic sections of $\mathcal Q$ such that $B_0$   does not contain the fixed points of $\tau$ and has at most negligible singularities. The double cover $\wt Y\to \mathcal Q$ is a $K3$ surface  with canonical singularities. In particular it has two $A_1$ points over the vertex $[0,0,0,1]\in \mathcal Q$. The involution $\tau$ can be lifted to a free involution $\wt \tau$ of $\wt Y$. The quotient surface $\wt Y/\wt \tau$ is an Enriques surface with canonical singularities, and by construction it is a double cover of $\mathcal D$ branched over the singular points $P_0,P_1,P_2$ and on the image $B$ of $B_0$.  The preimage of $P_0$ is an $A_1$ singular point, which gives a nodal curve $C$ on the minimal resolution $Y$ of $\wt Y/\wt \tau$; the preimage of the line joining $P_0$ and $P_1$ gives an elliptic half-pencil $E$ of $Y$ such that $EC=1$.

We now specialize this construction in order to get an Enriques surface of Godeaux quotient type. We take $B_0=D+\tau^*D$, where $D$ is a general quadric section of $\mathcal Q$. The curve $B_0$ has 8 nodes at the intersection points of $D$ and $\tau^*D$,  so in this case $\wt Y$ has 10 $A_1$ points, two occurring over the vertex of $\mathcal Q$ and eight  occurring over the nodes of $B$. These last eight points are an even set by Lemma \ref{lem:even}. 
As in the proof of Lemma \ref{lem:even} consider the bidouble cover  $h\colon Z \to \mathcal Q$ given by the relations $2L_1\equiv D$, $2L_2\equiv \tau^*D$, where $L_1=L_2=\OO_{\mathcal Q}(1)$. As in  \S\ref{ssec:bidouble} we denote by $G=\{1,g_1,g_2,g_3\}$ the Galois group of the bidouble cover and we assume that $g_1$, respectively $g_2$,  fixes the preimage of $D$, respectively $\tau^*D$, pointwise,  so  that $Z/g_3=\wt Y$. As explained in \S \ref{ssec:lift}, it is possible to lift $\tau$ to an automorphism $\rho$ of $Z$ and the group $\wt G<\Aut(Z)$ generated by the Galois group $G\cong \Z_2^2$ and by  $\rho$  is isomorphic to the dihedral group  $D_4$. The subgroup $G<\wt G$ contains two reflections  conjugate to one another  and the square of a rotation, so we may choose the lift $\rho$ of $\tau$ to be a rotation. Since $\tau$ switches $D$ and $\tau^*D$, the action of $\rho$ on $G$ by conjugation switches $g_1$ and $g_2$ and fixes $g_3$. It follows that $g_1$ and $g_2$ are reflections   and $g_3=\rho^2$. 
Now let $\wt \tau$ be the automorphism of $\wt Y=Z/\rho^2$ induced by $\rho$.  The fixed locus of $\rho^2$ on $Z$ is the set of 8 points lying over the nodes of $D+\tau^*D$.  Since $\rho$ acts freely on these points, it  follows that $\rho$ acts freely on $Z$ and $\wt \tau$ acts freely on $\wt Y$ (the fixed points of $\wt \tau$ correspond to solutions $z\in Z$ of $\rho z=z$ or $\rho z=\rho^2z$). Let $Y$ be the minimal resolution of  the surface $\wt Y/\wt \tau=Z/\rho$.   The surface $Y$ is a special Enriques surface that contains, besides  $C_5:=C$  as in the general case, four additional disjoint nodal curves $C_1,\dots, C_4$ arising from the 4 nodes of $\wt Y/\wt \tau$ that are the images of  the 8 nodes of $\wt Y$. Since the nodes of $\wt Y$ are an even set, by Lemma \ref{lem:div} either $C_1+\dots +C_4$ or $C_1+\dots +C_4+K_Y$ is even. Lemma \ref{lem:div-geo} tells us that the latter case occurs, and therefore $Y$ is an Enriques surface of Godeaux-quotient type. 

\end{ex} 

\begin{thm}\label{thm:main-enriques}
The general Enriques surface of Godeaux-quotient type  can be constructed as in Example \ref{ex:main-enriques}.  
\end{thm}
\begin{proof} Since $\Aut(\mathcal Q)$ has dimension 3, the construction gives a $5$-dimensional family of Enriques surfaces of Godeaux-quotient type and the statement follows by Corollary \ref{cor:irreducible}.
\end{proof}

\section{A construction of the general Godeaux surface with an Enriques  involution}\label{ssec:bidouble-godeaux}

We give an alternative description of the general Godeaux surface with an involution of Enriques type, that will be used in \S \ref{sec:deg}  to compute some stable degenerations. 
\smallskip

We keep the notation of the previous section (especially of Example \ref{ex:main-enriques}).  We take $B_1$ a general quadratic section of $\mathcal Q$, $B_2=\tau^*B_1$ and $B_3$ a  general  hyperplane section containing the two smooth fixed points $Q_1$ and $Q_2$  of $\tau$ (notice that $B_3$ is $\tau$-invariant).  
Consider the minimal resolution $\mathbb F_2\to \mathcal Q$, denote by $\Gamma$ the exceptional curve  and use the same letter to denote curves on $\mathcal Q$ and their pull-backs to $\mathbb F_2$. 
By \S \ref{ssec:bidouble} there exists a bidouble cover  $T_0\to \mathbb F_2$  with branch divisors $B_1,B_2,B_3+\Gamma$ and by \S\ref{ssec:lift} the involution of $\mathbb F_2$ induced by $\tau$ can be lifted to an automorphism of $T_0$.
 The  preimage of  $\Gamma$ is the disjoint union of two irreducible $-1$-curves. Contracting  these two curves, one obtains a bidouble cover $q\colon T\to \mathcal Q$, with $T$ smooth, with branch divisors $B_1$, $B_2$ and $B_3$, which is branched also on the vertex  $Q_0=[0,0,0,1]$ of $\mathcal Q$. 
 By the Hurwitz formula, one has $K_T\sim \frac 12 B_3$, hence $T$ is smooth minimal  of general type with $K^2_T=2$.
 The group  $\wt G<\Aut(T)$ generated by the Galois group $G=\{1, g_1,g_2, g_3\}\cong \Z_2^2$ of $q$ and by a lift of $\tau$ is isomorphic to $D_4$ (cf. \S \ref{ssec:lift}). Denote by $\rho\in D_4$ an element of order 4: then $\rho$ is a lift of $\tau$,   $\rho^2$ is an element of $G$ and commutes  with $\rho$. Since $\tau$ exchanges $B_1$ and $B_2$,  we have  $g_3=\rho^2$ and  $g_1$ and $g_2=g_1\rho^2$  are reflections. 
As in Example \ref{ex:main-enriques}, the surface $\wt Y:=T/\rho^2$ is a $K3$ surface with 10 nodes.

\begin{lem}\label{lem:bidouble-godeaux}
In the above setting:
\begin{enumerate}
\item $g_1\rho$ and $g_1\rho^3$  induce a fixed point free involution of $\wt Y$;
\item the surfaces $T/g_1\rho$ and $T/g_1\rho^3$ are Godeaux surfaces with an Enriques involution. 
\end{enumerate}
\end{lem}
\begin{proof} 
(i) 
There are two liftings of $\tau$ to $\wt Y$,  one induced by $\rho$ and  the other one induced by $g_1\rho$. We know (cf. Example \ref{ex:main-enriques}) that one of these  acts freely, while  the other one fixes 8 points. Assume for  contradiction that $\rho$ induces a fixed point free involution $\wt \tau$ and denote by  $Y$ the minimal desingularization of $\wt Y/\wt \tau$. By Example \ref{ex:main-enriques},  $Y$ is an Enriques surface of Godeaux quotient type; in particular $B+C_1+\dots +C_5$ is divisible by $2$ in $\Pic(Y)$, where we denote by $B$ the  strict transform of the image of $B_3$ and by $C_1,\dots C_5$ the nodal curves that arise from the resolution of  the images of the 10 nodes of $\wt Y$. 
On the other hand, arguing as we did at the end  of Example \ref{ex:main-enriques} we see that  $B+C_1+\dots +C_5+K_Y$ is divisible by 2 in $\Pic(Y)$. It follows that $K_Y$ is divisible by 2 in $\Pic(Y)$, a contradiction. So the fixed point free involution $\wt \tau$ of $\wt Y$ that lifts $\tau$ is induced by $g_1\rho$. Clearly, also $g_1\rho^3$ induces the same involution.

\medskip

(ii) By (i)  $g_1\rho$ is a fixed point free involution of $\wt Y$ and the same is true of the conjugate involution $g_1\rho^3$. The surfaces $S_1:=T/g_1\rho$ and $S_2:=T/g_1\rho^3$  are isomorphic;   they are  smooth minimal of general type with $K^2_{S_i}=1$ for $i=1,2$, hence they are Godeaux surfaces. The involution $\rho^2$ induces on $S_1$ and $S_2$  an Enriques involution with quotient $\wt Y/\wt \sigma$.
\end{proof}
\begin{prop}\label{prop:bidouble-godeaux}
The family of surfaces constructed as in Lemma \ref{lem:bidouble-godeaux}, (ii) contains  a dense  open subset of the family of  Godeaux surfaces with an Enriques involution.
\end{prop}

\begin{proof}
By Corollary \ref{cor:Godeaux-moduli}, it suffices to count dimensions. 
\end{proof}

\section{Stable degenerations of Godeaux surfaces with an Enriques involution}\label{sec:deg}
At the beginning of this section we recall   some facts on stable Godeaux surfaces.  Then we describe  some  examples, obtained by 
letting the branch divisors in the construction given in  \S  \ref{ssec:bidouble-godeaux} of the general Godeaux surfaces with an Enriques involution  acquire singularities or multiple components. 
\subsection{Non-normal Gorenstein stable Godeaux surfaces}\label{ssec:-stable-def}
The notion of stable surface  generalizes that  of (canonical model of) minimal surface  of general type in the same way as the notion of stable curve generalizes that of smooth curve of genus $>1$: there exists a projective coarse moduli space $\overline{\mathcal M_{a,b}}$ parametrizing  stable surfaces with fixed numerical invariants $K^2=a$ and $\chi=b$ and the moduli space  of surfaces of general type with  the same invariants is an open subset $\mathcal M_{a,b}\subset \overline{\mathcal M_{a,b}}$ (cf. \cite{Ale06} for an exposition of the theory of stable varieties and, more generally, of  stable pairs).

We recall the definition: a {\em stable surface} is a projective surface $S$ such that: 
\begin{itemize}
\item in the terminology of \cite{SMMP} the surface $S$ is {\em demi-normal}. This means  that $S$ satisfies condition $S_2$ of Serre  and there exists an open subset $S_0\subset S$ such that $S\setminus S_0$ is a finite set and for every $x\in S_0$ the point $x$ is either smooth or double crossings  (i.e.,  $S$ is locally isomorphic to $xy=0$ in the analytic or \'etale topology).
\item let $\bar S\to S$ be the normalization map and let $\bar D \subset \bar S$ be the {\em double locus}, that is, $\bar D$  the effective divisor defined by the conductor ideal sheaf; then $(\bar S, \bar D)$ is a log-canonical pair.
\item there exists an integer $m$ such that $\OO_S(mK_S)$ is an ample line bundle. 
\end{itemize}
If $S$ is a stable surface, we denote by $\nu(S)$ the {\em Cartier index} of $S$, namely the smallest $m>0$ such that $mK_S$ is Cartier. 

We call a stable surface with $K^2_S=\chi(S)=1$ a {\em stable Godeaux surface};   we say that $S$ is {\em classical} if it has at most rational double points, i.e., if it is the canonical model of a minimal smooth surface of general type  $Y$ with $K_Y^2=\chi(Y)=1$.  
We are mainly interested in the case in which  $S$ is  Gorenstein. Under this assumption,  one has $h^1(\OO_S)=h^2(\OO_S)=0$ (\cite[Prop.~4.2]{FPR2014b}) and the possibilities for the pair $(\bar S,\bar D)$  associated to a non-classical Godeaux surface $S$ are quite restricted:
 \begin{thm}[\cite{FPR2014a}, Thm.~3.7 and 4.1]\label{thm:class}
  Let $S$ be a non-classical stable Godeaux surface and let $(\bar S,\bar D)$ be the corresponding  log-canonical pair. If $S$ is Gorenstein, then one of the following cases occurs:
\begin{itemize}
\item[$(N)$] $S=\bar S$, namely $S$ is normal. Denote by $\epsi\colon \wt S\to S$ the minimal desingularization; in this case $\chi(\wt S)=0$ and the only non canonical singularity of $S$ is an elliptic singularity; 
\item[$(P)$] $\bar S=\pp^2$, $\bar D$ a quartic;
\item[$(dP)$] $\bar S$ is a del Pezzo surface of degree 1, with at most  canonical singularities, and $D\in |-2K_{\bar S}|$;
\item[$(E_+)$] $\bar S$ is the symmetric product of a curve $E$ of genus 1 and $\bar D$ is a stable curve of genus 2 which is a trisection of the Albanese map $\bar S\to E$. 
\end{itemize}
\end{thm}
\begin{remark} More precisely,  in  \cite{FPR14c} it is shown that in case (N) the surface  $\wt S$ is either the blow up of a bielliptic surface  at a point or a surface ruled over an elliptic curve and the bielliptic  case is completely classified. An example with $\wt S$ ruled   appears in \cite[Ex.~2.14]{lee00}; in \S \ref{ssection:deg} we give a new one. 

The non-normal stable Gorenstein  Godeaux surfaces of type $(dP)$ are described in \cite{RollenskeDP}, where it is shown that they form an irreducible component of the moduli space, hence in particular they are not smoothable. 

 The  non-normal stable Gorenstein  Godeaux surfaces of type $(P)$ and $(E_+)$ are classified   in \cite{FPR14c}. 
 \end{remark}
 
 Here we   recall the description of  one family of surfaces of type $(P)$ such that the general surface in the family has an involution. These surfaces are obtained in \S \ref{ssection:deg} as specializations of the Godeaux surfaces with an Enriques involution, and therefore they are smoothable (cf. Proposition \ref{prop:Psmooth}). 
 
 \begin{ex} \label{ex:P}
 Let $P_1,\dots P_4\in \pp^2$ be independent points and let $\phi\colon \pp^2\to\pp^2$ be the projective automorphism such that $\phi(P_i)=P_{i+1}$ for $1\le i\le 4$ (indices are taken modulo $4$).  The automorphism $\phi$ induces  on the pencil $\cF$ of conics through $P_1, \dots P_4$  an involution that fixes the reducible conic $L(P_1,P_3)+L(P_2,P_4)$ and a smooth conic $C_0\in \cF$.  We take $\bar S=\pp^2$ and $\bar D=C+\phi_* C$, where $C\in\cF\setminus \{C_0\}$ is a smooth conic. By \cite[Thm.~5.13]{SMMP} (cf. also \cite[Thm.~3.2]{FPR2014a} for the Gorenstein condition) in order to construct a Gorenstein stable surface with $K^2=1$  with normalization  $(\bar S, \bar D)$, one has to give an involution $\iota$ of the normalization $C\sqcup \phi_*C$ of $\bar D$ with the property that $\iota$ acts freely on the   eight preimages of $P_1,\dots P_4$.  We take $\iota$ to be the involution that exchanges $C$ and $ \phi_*C$ and identifies $C$ with $ \phi_*C$ via $\phi$. One has $\chi(S)=1$  by  \cite[Prop.~3.4]{FPR2014a}. The involution $\phi^2$ of $\pp^2$ commutes with $\iota$ and therefore it induces an involution  of $S$ (cf. \cite[\S3.B]{FPR2014b}). 
 \end{ex}

\subsection{Degenerating Godeaux surfaces with an Enriques involution}\label{ssection:deg}

A way of obtaining stable degenerations of  a Godeaux surface with an Enriques involution is to apply the construction described in \S \ref{ssec:bidouble-godeaux} relaxing the assumption that  the branch divisors be general. Keeping  the notation of   \S \ref{ssec:bidouble-godeaux}, we take $B_1$ a divisor in $|\OO_{\cQ}(2)|$, $B_2=\tau^*B_1$, $B_3$ a hyperplane section through $Q_1$ and $Q_2$ such  that  the pair $(\cQ, \frac 12(B_1+B_2+B_3))$ is log-canonical and we construct the bidouble cover $T\to \cQ$ with branch data $B_1,B_2, B_3$. Observe that $\rho$ induces an isomorphism between the quotient surfaces  $T/g_1\rho$ and $T/g_1\rho^3$; we abuse notation and refer to either  of these surfaces as to  $S$.
By  Proposition  \ref{prop:bidouble-godeaux} the surface $S$  is a degeneration of the general  Godeaux surfaces with an Enriques involution. The next result shows that it is indeed a stable degeneration:

 \begin{lem}\label{lem:gorenstein} Consider the setup and notation  of \S \ref{ssec:bidouble-godeaux}
  and assume 
  that  the pair $\left(\cQ, \frac 1 2(B_1+B_2+B_3)\right)$ is log-canonical.\\
 Then:
 \begin{enumerate}
 \item $T$ is a stable  Godeaux surface  with  $\nu(T)=1$ or $2$. If $Q_0\notin B_1+B_2+B_3$ and $B_1\cap B_2\cap B_3=\emptyset$, then $T$ is Gorenstein.
   \item  $S$ is a  stable  Godeaux surface    such that $\nu(S)$ divides $2\nu(T)$; 
\item if $B_1+B_2$ does not contain any of the fixed points $Q_0,Q_1,Q_2$ of $\tau$  on $\cQ$, then  $T\to S$ is an \'etale morphism, and in particular $\nu(S)=\nu(T)$. 
\end{enumerate}
 \end{lem} 
 \begin{proof} (i)  The cover $T\to \cQ$ is demi-normal by  \cite[Thm.~1.9]{rita-valery_non-normal}.
By Prop.~2.5, ibidem,  the surface $T$ is slc and $2K_{T}$ is the pull back of $2K_{\cQ}+(B_1+B_2+B_3)= H$, where $H$ is the hyperplane section of $\cQ$.
Hence $K_{T}$ is ample and $2$-Cartier.  

If $Q_0\notin B_1+B_2+B_3$, then $T$ is smooth (hence Gorenstein) over $Q_0$;  if  $B_1\cap B_2\cap B_3=\emptyset$ then locally over every smooth point of $\cQ$, $T\to\cQ$ is the composition of two flat double covers and therefore it is Gorenstein. 
\medskip

(ii)  Since $g_1\rho$ lifts $\tau$, that has only isolated fixed points, the quotient  map $T\to S$ is unramified in codimension 1, hence again by  \cite[Prop.~2.5]{rita-valery_non-normal} we have that  $S$ is an slc surface and  $K_{S}$ is ample, since it pulls back to $K_{T}$. In addition, the argument in the proof of \cite[Lem.~2.3]{rita-valery_non-normal} shows that  $\nu(S)$ divides $2\nu(T)$. The fact that $K^2_S=\chi(\OO_S)=1$ follows from the fact that $S$ can be obtained as a flat limit of smooth Godeaux  surfaces and so  $S$ is a stable Godeaux surface.\medskip

(iii) It is enough to show that the involution of $Y:=S/\rho^2$ induced by $g_1\rho$ is base point free. If $B_1$ and $B_2$ are general, $Y$ is a nodal $K3$ surface and the involution induced by $\rho$ fixes all the preimages of $Q_0$, $Q_1$ and $Q_2$ (cf. proof of Lemma \ref{lem:bidouble-godeaux}). By continuity, the involution induced by $\rho$  fixes the preimages of the fixed points of $\tau$ for every choice of $B_1$ and $B_2$. Since $g_1$ induces the covering involution of $Y\to \cQ$, if $Y\to \cQ$ is unramified over $Q_0,Q_1,Q_2$, then the involution  of $Y$ induced by $g_1\rho$ acts freely on the preimages of $Q_0,Q_1,Q_2$, hence it acts freely on $Y$. 

 \end{proof}
 \subsection{Examples of degenerations}
We examine now some instances of the situation of \S \ref{ssection:deg}. Recall that $T$ (and $S$)   is normal iff $B_1+B_2+B_3$ is a reduced divisor;  in general,  the normalization $\bar T$ of $T$ is a  bidouble cover of $\cQ$  whose construction is described in  \cite[\S3]{rita-abel}. For the description of the possible singularities of $T$ we refer the reader to \cite[\S3]{rita-valery_non-normal}.
\bigskip

(1) {\em ${B_1}$ and ${B_2}$  intersect at two points $R_1$, $R_2$ that are double points of both.} \\
An example of this type can be constructed as follows.  Choose  $R_1\in \cQ$ general and set $R_2=\tau(R_1)$. If $H_1,\dots H_4$ are general hyperplane sections containing $R_1$ and $R_2$, then $H_1+H_2$ and $H_3+H_4$ span a pencil of quadric sections. 
  We take  $B_1$  a general element of this pencil, so that $B_1$ has ordinary double points at $R_1$ and $R_2$ and is smooth elsewhere; as usual, we set $B_2=\tau^*B_1$. Since $B_1B_2=8$, the divisor $B_1+B_2$ has  ordinary quadruple points at $R_1$ and $R_2$. We assume that $B_3$ is general; by Lemma \ref{lem:gorenstein}, $T$ and $S$ are both Gorenstein.
The surface $T$ has two elliptic singularities $U_1$ and $U_2$ of degree 4 over $R_1$ and $R_2$ (cf. Table 1 of  \cite[\S3]{rita-valery_non-normal}). These singularities  map in $S$ to one  elliptic singularity of the same type, hence  $1=\chi(S)=\chi(\wt  S)+1$, where $\wt S$ is the minimal desingularization of $S$.   
The minimal desingularization $\wt T\to T$  is obtained  by blowing up $\cQ$ at $R_1$ and $R_2$ and taking base change and normalization; the exceptional curves of the blow-up $\hat\cQ\to \cQ$ are not  contained in the branch locus of $\wt T\to\hat \cQ$. 
Therefore the  strict transforms on $\hat \cQ$ of the plane sections of $\cQ$ through $R_1$ and $R_2$   meet the branch locus of $\wt T\to\hat \cQ$ only at two points,  and  so their preimages  in $T$ are   pairs of rational curves.  So $T$ is ruled and therefore $S$ and  $\wt S$ are  ruled, too. Since $\chi(\wt S)=0$, 
the surface $\wt S$ is ruled over an elliptic curve.

This is a new example of case (N) of Theorem \ref{thm:class} with $S$ ruled; the other known example (cf. \cite[Ex.~2.14]{lee00}) has an elliptic singularity of degree 3. 
\bigskip

(2)  {\em ${B_1}=2H$, with $H$ a general hyperplane section.}\\
We have $B_2=2\tau^*H$ and we  take $B_3$ general; by Lemma \ref{lem:gorenstein}, $T$ and $S$ are both Gorenstein. In this case, the surface $\wt Y=T/\rho^2$ is the union of two copies of $\cQ$ glued along the curve $H+\tau^*H$. The surface $T$ is non-normal and has two irreducible components, both isomorphic to the double cover of $\cQ$ branched on  the plane section $B_3$ and on the vertex $Q_0$ of $\cQ$, and therefore both  isomorphic to $\pp^2$. 
By Lemma \ref{lem:gorenstein}, the surface $S$ is Gorenstein and therefore irreducible, since $K^2_{S}=1$.   So $g_1\rho$ permutes the two components of $T$ and the normalization $\bar S$ of $S$ is isomorphic to $\pp^2$, namely $S$ is  as in case  $(P)$ of Theorem \ref{thm:class}.
The surface $\bar S=\pp^2$ can be naturally identified with one of the irreducible component of $T$; we denote by $\pi\colon \bar S\to \cQ$  the degree 2 map induced by this identification.   The double locus $\bar D\subset \bar S$ is the union of two conics, $C_1:=\pi^*H$ and $C_2:=\pi^*(\tau^*H)$,  that are identified with one another by the involution $\iota$ of $C_1\sqcup C_2$ induced by the map $\bar S\to S$. 

We claim that the surface $S$ belongs to the family constructed in  Example \ref{ex:P}.
   Let $R_1,R_2$  be the intersection points of $H$ and $\tau^*H$ in $\cQ$ and write $\pi\inv(R_1)=\{P_1, P_3\}$ and $\pi\inv(R_2)=\{P_2,P_4\}$. 
  The points $P_1, \dots P_4$ are the base points  of the pencil of conics spanned by $C_1$ and $C_2$. By construction,  the 
   involution $\iota$  of $C_1\sqcup C_2$   lifts the involution  of $H+\tau^*H$ given by $\tau$. 
    The involution $\tau$ lifts to an automorphism of $\bar S$ that exchanges the sets $\pi\inv(R_1)$ and $\pi\inv(R_2)$ and exchanges the conics $C_1$ and $C_2$. Elementary arguments on pencils of plane conics show that such a map  is 
     either the that  automorphism  $\phi$ that  induces a cyclic permutation  of $P_1,\dots P_4$ or its inverse $\phi^3$. So, possibly up to relabelling the $P_i$, the involution $\iota$ of $C_1\sqcup C_2$ induced by the normalization map $\bar S\to S$ switches $C_1$ and $C_2$ and identifies $C_1$ with $C_2$ via $\phi$. 
 Since letting  $H$ vary in the pencil of plane sections through $R_1$ and $R_2$ we can obtain any conic in the pencil  spanned by $C_1$ and $C_2$, we have proven the following:
     
 \begin{prop} \label{prop:Psmooth}
 The  surfaces  in the family of Example \ref{ex:P} are  smoothable.
 \end{prop}
\bigskip

(3) {\em $B_1$  and $B_2$ have a common component  which is a hyperplane section.}\\
Take $B_1=H_0+H_1$, where $H_0$ is a $\tau$-invariant hyperplane section and $H_1$ is a general one, so that $B_2=H_0+\tau^*H_1$,  and take $B_3$ general. Assume that $H_0$ does not contain the vertex $Q_0$ of $\cQ$, hence $H_0$  contains the two smooth fixed points $Q_1$ and $Q_2$ of $\tau$. By Table 2 of \cite[\S3]{rita-valery_non-normal} the singularities of $T$ over $Q_1$ and $Q_2$ are not Gorenstein, so $\nu(T)=2$ by Lemma \ref{lem:gorenstein}, and it follows that  $S$ is not Gorenstein either.

By \cite[\S3]{rita-abel}, the normalization $\bar T$ of $T$ is the bidouble cover of $\cQ$ branched on $H_1$, $\tau^*H_1$, $B_3+H_0$ and the vertex $Q_0$ of $\cQ$. The surface $\bar T$ has a pair of  singular points of type $A_1$ over $Q_1$ and over $Q_2$ and is smooth elsewhere.  By the Hurwitz formula the canonical class $K_{\bar T}$ is numerically equivalent to $0$. Taking base change of $\bar T\to \cQ$  with the minimal resolution $\mathbb F_2\to \cQ$ one obtains a  flat bidouble cover $T_0\to \mathbb F_2$.  The standard formulae for double covers give $p_g(T_0)=q(T_0)=0$, hence $\bar T$ is an Enriques surface with four nodes. The involution $g_1\rho$ of $T$ 
induces an involution of the minimal desingularization $\wt T$ of $T$, whose fixed locus is contained in the preimages of the points $Q_0,Q_1,Q_2\in \cQ$. The preimage of $Q_0$ consists of two smooth points, while the preimage of $\{Q_1,Q_2\}$ is the disjoint union of four nodal curves. Assume that one of these nodal curves is preserved by $g_1\rho$;  then a local computation shows that this curve is not fixed pointwise by $g_1\rho$. Summing up,  the fixed locus of $g_1\rho$ on $\wt T$ is finite. It follows that the quotient surface $\wt T/g_1\rho$ is again an Enriques surface, and so is $\bar S$, since it is birational to $\wt T/g_1\rho$.

This example shows that if we remove the assumption that $S$ is Gorenstein, then Theorem \ref{thm:class} does not hold any more. 
\bigskip

(4) {\em $B_1=H_1+2F_1$, where $H_1$ is a general hyperplane section and $F_1$ is a general ruling of $\mathcal Q$.}\\
Set $H_2=\tau^*H_1$, $F_2=\tau^*F_1$, so that $B_2=H_2+2F_2$. The surface $T$ is singular above $F_1$ and $F_2$. The normalization $\bar T$ of $T$ is a bidouble cover of $\cQ$ branched on the three hyperplane sections $H_1$, $H_2$ and $B_3$, so $K_{\bar T}$ is numerically equivalent to the the pull-back of $-\frac 12 H_1 $, and  $\bar T$ is a del Pezzo surface of degree 2.  The map $\bar T\to\cQ$ is unramified  over the vertex $Q_0$, hence the singularities of $\bar T$ are four  points $U_1,U_2,U_3,U_4$ of type $A_1$ occurring above $Q_0$. The elements $g_1,g_2=g_1\rho^2, \rho^2$ of $D_4$ act  on $U_1,U_2,U_3,U_4$ switching them in pairs, so $\rho$ acts as a cyclic permutation of order $4$ and $g_1\rho$ switches, say, $U_1$ and $U_3$ and fixes $U_2$ and $U_4$.
Looking at the minimal resolution $\wt T$ of $\bar T$, one sees that $g_1\rho$ has two isolated fixed points on each of the   nodal curves corresponding to $U_2$ and $U_4$, hence the fixed locus of  $g_1\rho$ on $\wt T$ is a finite set and the quotient surface $\bar S=\bar T/g_1\rho$ has canonical singularities (the images of $U_2$ and $U_4$ are points of type $A_3$).  Hence $\bar S$ is a del Pezzo surface of degree 1.

The double locus $D_T\subset \bar T$ is the preimage of $F_1+F_2$: it consists of two smooth rational curves $\Gamma_1$ and $\Gamma_2$ meeting transversely at $U_1,\dots U_4$ and it is an antibicanonical curve. The double locus $\bar D\subset \bar S$ is the image of $D_T$: it is an irreducible curve with $p_a=1$, since it is smooth at the images of $U_2$ and $U_4$ and  it has a node at the image point  of $U_1$ and $U_3$. The curve $\bar D$ is  numerically equivalent to an antibicanonical curve, since it pulls back to $D_T$, but it is not Cartier since it is smooth at the $A_3$ points of $\bar S$ (notice also the failure of the usual adjunction formula), hence it is not in $|-2K_{\bar S}|$.    So this case is different from  case $(dP)$ of Theorem \ref{thm:class}. In fact, the surface $S$ is not Gorenstein, since $K_{\bar S}+\bar D$ is not Cartier.

\bibliographystyle{alpha}
 \bibliography{rita}{}

 \end{document}